\documentclass[a4paper,11pt]{amsart}
\usepackage{amsmath,amssymb,amsfonts,latexsym}

\newtheorem{theorem}{Theorem}[section]

\input{tcilatex}

\begin{document}
\title[Inverse problem for retarded differential operator]{Traces and
inverse nodal problems for a class of Sturm-Liouville operators with
retarded argument}
\author{Erdo\u{g}an \c{S}en}
\address{Department of Mathematics, Tekirdag Nam\i k Kemal University,
59030, Tekirda\u{g}, Turkey}
\email{erdogan.math@gmail.com}

\begin{abstract}
In this study, we investigate the traces and solutions of inverse nodal
problems of discontinuous Sturm-Liouville operators with retarded argument
and with a finite number of transmission conditions.


\vspace{2mm}

\end{abstract}

\maketitle



\section{Introduction}


Sturm-Liouville problems with transmission conditions (also known as
interface conditons, discontinuity conditions, impulse effects) arise in
many applications. Amongst the applications are thermal conduction in a thin
laminated plate made up of layers of different materials and diffraction
problems. The main goal of this paper is to extent and generalize some
approaches and results of this kind of boundary value problems to similar
types of problems but with eigenvalue-parameter dependent boundary
conditions. To this aim, we calculate regularized traces and solve inverse
nodal problems of a class of Sturm-Liouville operators with retarded
argument and with a finite number of transmission conditions.

Namely, we consider the boundary value problem for the differential equation%
\begin{equation}
y^{\prime \prime }(t)+q(t)y(t-\Delta (t))+\mu ^{2}y(t)=0  \tag{1}
\end{equation}%
on $\Omega =\cup _{j=0}^{m}\Omega _{j}$ ($\Omega _{0}=\left[ 0,\theta
_{1}\right) ,$ $\Omega _{i}=\left( \theta _{i},\theta _{i+1}\right) $ $%
\left( i=\overline{1,m-1}\right) ,$ $\Omega _{m}=\left( \theta _{m},\pi %
\right] $) with boundary conditions%
\begin{equation}
\alpha _{1}^{-}y(0)-\alpha _{2}^{-}y^{\prime }\left( 0\right) +\mu \left(
\alpha _{1}^{+}y(0)-\alpha _{2}^{+}y^{\prime }\left( 0\right) \right) =0, 
\tag{2}
\end{equation}%
\begin{equation}
\beta _{1}^{-}y(\pi )-\beta _{2}^{-}y^{\prime }\left( \pi \right) +\mu
\left( \beta _{1}^{+}y(\pi )-\beta _{2}^{+}y^{\prime }\left( \pi \right)
\right) =0  \tag{3}
\end{equation}%
and transmission conditions%
\begin{equation}
y\left( \theta _{i}-\right) -\delta _{i}y(\theta _{i}+)=0,  \tag{4}
\end{equation}%
\begin{equation}
y^{\prime }\left( \theta _{i}-\right) -\delta _{i}y^{\prime }(\theta
_{i}+)=0,  \tag{5}
\end{equation}%
where the real-valued function $q\left( t\right) $ is continuous in $\Omega $
and has finite limits%
\begin{equation*}
q\left( \theta _{i}\pm \right) =\lim_{t\rightarrow \theta _{i}\pm }q\left(
t\right) ,
\end{equation*}%
the real-valued function $\Delta (t)\geq 0$ is continuous in $\Omega $ and
has finite limits%
\begin{equation*}
\Delta \left( \theta _{i}\pm \right) =\lim_{t\rightarrow \theta _{i}\pm
}\Delta \left( t\right) ,
\end{equation*}%
\textit{if}$\ t\in \Omega _{1}$ then $t-\Delta (t)\geq 0$; \textit{if} $t\in
\Omega _{i}$ then $t-\Delta (t)\geq \theta _{i}$ $(i=\overline{2,m})$; $\mu $
is a spectral parameter; $\alpha _{j}^{\pm },\beta _{j}^{\pm }$ $\left(
j=1,2\right) ,$ $\delta _{i}\neq 0,$ $\theta _{i}$ $(i=\overline{1,m})$ are
arbitrary real numbers such that $\theta _{0}=0<\theta _{1}<\theta
_{2}<...\theta _{m}<\theta _{m+1}=\pi $ and $\alpha _{2}^{+}\beta
_{2}^{+}\neq 0$.

We want to also note that some eigenvalue problems encountered in areas of
data mining requires the investigation of traces of operators and matrices
optimizing the certain properties of given input high-dimensional data (see
[1,2]).


\section{The regularized trace}


Let $\varphi _{1}(t,\mu )$ be a solution of Eq. (1) on $\left[ 0,\theta _{1}%
\right] ,$ satisfying the initial conditions%
\begin{equation}
\varphi _{1}\left( 0,\mu \right) =\mu \alpha _{2}^{+}+\alpha _{2}^{-}\text{ 
\textit{and} }\varphi _{1}^{\prime }\left( 0,\mu \right) =\mu \alpha
_{1}^{+}+\alpha _{1}^{-}\text{.}  \tag{6}
\end{equation}%
The conditions (6) define a unique solution of Eq. (1) on $\Omega _{1}\cup
\theta _{1}$ (see [3,4]).

After defining the above solution, then we will define the solution $\varphi
_{i}\left( t,\mu \right) $ of Eq. (1) on $\Omega _{i}\cup \left\{ \theta
_{i},\theta _{i+1}\right\} $ $\left( i=\overline{2,m}\right) $ by means of
the solution $\varphi _{1}\left( t,\mu \right) $ using the initial conditions%
\begin{equation}
\varphi _{i+1}\left( \theta _{i},\mu \right) =\frac{\varphi _{i}\left(
\theta _{i},\mu \right) }{\delta _{i}}\text{ \textit{and}}\quad \varphi
_{i+1}^{\prime }(\theta _{i},\mu )=\frac{\varphi _{i}^{\prime }(\theta
_{i},\mu )}{\delta _{i}}\text{.}  \tag{7}
\end{equation}%
The conditions (7) define a unique solution of Eq. (1) on $\Omega _{i}\cup
\left\{ \theta _{i},\theta _{i+1}\right\} $ $\left( i=\overline{2,m}\right) $%
. (see [5-10]).

Consequently, the function $\varphi \left( t,\mu \right) $ is defined on $%
\Omega $ by the equality%
\begin{equation*}
\varphi (t,\mu )=\varphi _{i+1}(t,\mu ),\text{ }t\in \Omega _{i}\text{ }%
\left( i=\overline{0,m}\right)
\end{equation*}%
is a solution of (1) on $\Omega $, which satisfies one of the boundary
conditions and the transmission conditions (4)-(5) Then the following
integral equations hold:\ 
\begin{align}
\varphi _{1}(t,\mu )& =\left( \mu \alpha _{2}^{+}+\alpha _{2}^{-}\right)
\cos \mu t-\frac{\mu \alpha _{1}^{+}+\alpha _{1}^{-}}{\mu }\sin \mu t  \notag
\\
& -\frac{1}{\mu }\int\limits_{0}^{{t}}q\left( \tau \right) \sin \mu \left(
t-\tau \right) \varphi _{1}\left( \tau -\Delta \left( \tau \right) ,\mu
\right) d\tau ,  \tag{8}
\end{align}%
\begin{align}
\varphi _{i+1}(t,\mu )& =\frac{1}{\delta _{i}}\varphi _{i}\left( \theta
_{i},\mu \right) \cos \mu \left( t-\theta _{i}\right) +\frac{\varphi
_{i}^{\prime }\left( \theta _{i},\mu \right) }{\mu \delta _{i}}\sin \mu
\left( t-\theta _{i}\right)  \notag \\
& -\frac{1}{\mu }\int\limits_{\theta _{i}}^{{t}}q\left( \tau \right) \sin
\mu \left( t-\tau \right) \varphi _{i+1}\left( \tau -\Delta \left( \tau
\right) ,\mu \right) d\tau ,\text{ }\left( i=\overline{1,m}\right) .  \tag{9}
\end{align}%
Solving the equations (8)-(9) by the method of successive approximation, we
obtain the following asymptotic equalities for $\left\vert \mu \right\vert
\rightarrow \infty :$%
\begin{equation*}
\varphi _{1}(t,\mu )=\mu \alpha _{2}^{+}\cos \mu t+\alpha _{2}^{-}\cos \mu
t-\alpha _{1}^{+}\sin \mu t-\frac{\alpha _{2}^{+}}{2}\int\limits_{0}^{t}q(%
\tau )\sin \mu \left( t-\Delta \left( \tau \right) \right) d\tau
\end{equation*}%
\begin{equation*}
-\frac{\alpha _{2}^{+}}{2}\int\limits_{0}^{t}q(\tau )\sin \mu \left( t-2\tau
+\Delta \left( \tau \right) \right) d\tau +\frac{1}{\mu }\left[ \alpha
_{1}^{-}\sin \mu t-\frac{\alpha _{2}^{-}}{2}\int\limits_{0}^{t}q(\tau )\sin
\mu \left( t-\Delta \left( \tau \right) \right) d\tau \right.
\end{equation*}%
\begin{equation*}
-\frac{\alpha _{2}^{-}}{2}\int\limits_{0}^{t}q(\tau )\sin \mu \left( t-2\tau
+\Delta \left( \tau \right) \right) d\tau +\frac{\alpha _{1}^{+}}{2}%
\int\limits_{0}^{t}q(\tau )\cos \mu \left( t-\Delta \left( \tau \right)
\right) d\tau
\end{equation*}%
\begin{equation}
\left. -\frac{\alpha _{1}^{+}}{2}\int\limits_{0}^{t}q(\tau )\cos \mu \left(
t-2\tau +\Delta \left( \tau \right) \right) d\tau \right] +O\left( \frac{1}{%
\mu ^{2}}\right) ,  \tag{10}
\end{equation}%
\begin{equation*}
\varphi _{i+1}(t,\mu )=\frac{1}{\prod\limits_{i=1}^{m}\delta _{i}}\left[ \mu
\alpha _{2}^{+}\cos \mu t+\left( \alpha _{2}^{-}\cos \mu t-\alpha
_{1}^{+}\sin \mu t\right) \right.
\end{equation*}%
\begin{equation}
\left. -\frac{\alpha _{2}^{+}}{2}\int\limits_{0}^{t}q(\tau )\sin \mu \left(
t-\Delta \left( \tau \right) \right) d\tau -\frac{\alpha _{2}^{+}}{2}%
\int\limits_{0}^{t}q(\tau )\sin \mu \left( t-2\tau +\Delta \left( \tau
\right) \right) d\tau \right] +O\left( \frac{1}{\mu }\right) .  \tag{11}
\end{equation}

Differentiating (10)-(11) with respect to $t$, we get%
\begin{equation*}
\varphi _{1}^{\prime }(t,\mu )=-\mu ^{2}\alpha _{2}^{+}\sin \mu t+\mu \left[
-\alpha _{2}^{-}\sin \mu t-\alpha _{1}^{+}\cos \mu t-\frac{\alpha _{2}^{+}}{2%
}\int\limits_{0}^{t}q(\tau )\cos \mu \left( t-\Delta \left( \tau \right)
\right) d\tau \right.
\end{equation*}%
\begin{equation*}
\left. -\frac{\alpha _{2}^{+}}{2}\int\limits_{0}^{t}q(\tau )\cos \mu \left(
t-2\tau +\Delta \left( \tau \right) \right) d\tau \right] +\alpha
_{1}^{-}\cos \mu t
\end{equation*}%
\begin{equation*}
-\frac{\alpha _{2}^{-}}{2}\int\limits_{0}^{t}q(\tau )\cos \mu \left(
t-\Delta \left( \tau \right) \right) d\tau -\frac{\alpha _{2}^{-}}{2}%
\int\limits_{0}^{t}q(\tau )\cos \mu \left( t-2\tau +\Delta \left( \tau
\right) \right) d\tau
\end{equation*}%
\begin{equation}
-\frac{\alpha _{1}^{+}}{2}\int\limits_{0}^{t}q(\tau )\sin \mu \left(
t-\Delta \left( \tau \right) \right) d\tau +\frac{\alpha _{1}^{+}}{2}%
\int\limits_{0}^{t}q(\tau )\sin \mu \left( t-2\tau +\Delta \left( \tau
\right) \right) d\tau +O\left( \frac{1}{\mu }\right) ,  \tag{12}
\end{equation}%
\begin{equation}
\varphi _{i+1}^{\prime }(t,\mu )=-\frac{1}{\prod\limits_{i=1}^{m}\delta _{i}}%
\left\{ \alpha _{2}^{+}\mu ^{2}\sin \mu t+\mu \left[ -\alpha _{2}^{-}\sin
\mu t-\alpha _{1}^{+}\cos \mu t-\frac{\alpha _{2}^{+}}{2}\int%
\limits_{0}^{t}q(\tau )\cos \mu \left( t-\Delta \left( \tau \right) \right)
d\tau \right. \right.  \notag
\end{equation}%
\begin{equation}
\left. \left. -\frac{\alpha _{2}^{+}}{2}\int\limits_{0}^{t}q(\tau )\cos \mu
\left( t-2\tau +\Delta \left( \tau \right) \right) d\tau \right] \right\}
+O\left( 1\right) ,\text{ }i=\overline{1,m}.  \tag{13}
\end{equation}%
The solution $\varphi (t,\mu )$ defined above is a nontrivial solution of
(1) satisfying conditions (2) and (4)-(5). Putting$\>\varphi (t,\mu )\>$into
(3), we get the characteristic equation 
\begin{equation}
\Xi (\mu )\equiv \left( \mu \beta _{1}^{+}+\beta _{1}^{-}\right) \varphi
(\pi ,\mu )-\left( \mu \beta _{2}^{+}+\beta _{2}^{-}\right) \varphi ^{\prime
}(\pi ,\mu )=0\text{.}  \tag{14}
\end{equation}%
The set of eigenvalues of boundary value problem (1)-(5) coincides with the
set of the squares of roots of (14), and eigenvalues are simple. From
(10)-(14), we obtain%
\begin{equation*}
\Xi (\mu )\equiv \frac{\mu ^{3}\alpha _{2}^{+}\beta _{2}^{+}}{%
\prod\limits_{i=1}^{m}\delta _{i}}\sin \mu \pi +\frac{\mu ^{2}}{%
\prod\limits_{i=1}^{m}\delta _{i}}\left\{ \left[ \alpha _{2}^{+}\beta
_{1}^{+}+\alpha _{1}^{+}\beta _{2}^{+}\right. \right.
\end{equation*}%
\begin{equation*}
\left. +\frac{\alpha _{2}^{+}\beta _{2}^{+}}{2}\int\limits_{0}^{\pi }q(\tau
)\cos \mu \Delta (\tau )d\tau +\frac{\alpha _{2}^{+}\beta _{2}^{+}}{2}%
\int\limits_{0}^{\pi }q(\tau )\cos \mu \left( 2\tau -\Delta (\tau )\right)
d\tau \right] \cos \mu \pi
\end{equation*}%
\begin{equation*}
+\left[ \alpha _{2}^{-}\beta _{2}^{+}+\alpha _{2}^{+}\beta _{2}^{-}+\frac{%
\alpha _{2}^{+}\beta _{2}^{+}}{2}\int\limits_{0}^{\pi }q(\tau )\sin \mu
\Delta (\tau )d\tau \right.
\end{equation*}%
\begin{equation*}
\left. \left. +\frac{\alpha _{2}^{+}\beta _{2}^{+}}{2}\int\limits_{0}^{\pi
}q(\tau )\sin \mu \left( 2\tau -\Delta (\tau )\right) d\tau \right] \sin \mu
\pi \right\} +O(\mu )
\end{equation*}%
which is deduced to%
\begin{equation}
\Xi \left( \mu \right) =\frac{\mu ^{3}\alpha _{2}^{+}\beta _{2}^{+}}{%
\prod\limits_{i=1}^{m}\delta _{i}}\sin \mu \pi +\frac{\mu ^{2}}{%
\prod\limits_{i=1}^{m}\delta _{i}}\left\{ \left[ \alpha _{2}^{+}\beta
_{1}^{+}+\alpha _{1}^{+}\beta _{2}^{+}\right. \right.  \notag
\end{equation}%
\begin{equation*}
\left. +\alpha _{2}^{+}\beta _{2}^{+}\left( U^{+}(\mu )+V^{+}(\mu \right) 
\right] \cos \mu \pi
\end{equation*}%
\begin{equation}
\left. +\left[ \alpha _{2}^{-}\beta _{2}^{+}+\alpha _{2}^{+}\beta
_{2}^{-}+\alpha _{2}^{+}\beta _{2}^{+}\left( U^{-}(\mu )+V^{-}(\mu \right) %
\right] \sin \mu \pi \right\} +O(\mu ).  \tag{15}
\end{equation}%
Here,%
\begin{equation*}
U^{+}\left( \mu \right) =\frac{1}{2}\int\limits_{0}^{\pi }q(\tau )\cos
\left( \mu \Delta \left( \tau \right) \right) d\tau ,\text{ \ \ \ \ }%
U^{-}\left( \mu \right) =\frac{1}{2}\int\limits_{0}^{\pi }q(\tau )\sin
\left( \mu \Delta \left( \tau \right) \right) d\tau ,
\end{equation*}%
\begin{equation*}
V^{+}\left( \mu \right) =\frac{1}{2}\int\limits_{0}^{\pi }q(\tau )\cos
\left( \mu \left( 2\tau -\Delta \left( \tau \right) \right) \right) d\tau ,%
\text{ \ \ \ \ }V^{-}\left( \mu \right) =\frac{1}{2}\int\limits_{0}^{\pi
}q(\tau )\sin \left( \mu \left( 2\tau -\Delta \left( \tau \right) \right)
\right) d\tau .
\end{equation*}%
Define%
\begin{equation}
\Xi _{0}\left( \mu \right) \equiv \frac{\mu ^{3}\alpha _{2}^{+}\beta _{2}^{+}%
}{\prod\limits_{i=1}^{m}\delta _{i}}\sin \mu \pi ,  \tag{16}
\end{equation}%
and denote by $\mu _{\pm n}^{0},n\in 
\mathbb{Z}
,$ the zeros of the function $\Xi _{0}\left( \mu \right) $, except that zero
is multiplicity 4; then $\mu _{\pm 0}^{0}=\mu _{\pm 1}^{0}=0$ and

\begin{equation*}
\mu _{n}^{0}=\left\{ 
\begin{array}{c}
n-1,\text{ }n\geq 1, \\ 
n+1,\text{ }n\leq -1.%
\end{array}%
\right.
\end{equation*}

Denote by $C_{n}$ the circle of radius, $0<\varepsilon <\frac{1}{2},$
centered at the origin $\mu _{n}^{0}$ and by $\Gamma _{N_{0}}$ the
counterclockwise square contours with four vertices%
\begin{eqnarray*}
A &=&\left( N_{0}-1+\varepsilon \right) \left( 1-i\right) ,\text{ \ \ }%
B=\left( N_{0}-1+\varepsilon \right) \left( 1+i\right) , \\
C &=&\left( N_{0}-1+\varepsilon \right) \left( -1+i\right) ,\text{ \ \ }%
D=\left( N_{0}-1+\varepsilon \right) \left( -1-i\right) ,
\end{eqnarray*}%
where $i=\sqrt{-1}$ and $N_{0}$ is a natural number. Obviously, if $\mu \in
C_{n}$ or $\mu \in \Gamma _{N_{0}},$ then $\left\vert \Xi _{0}\left( \mu
\right) \right\vert \geq M\left\vert \mu \right\vert e^{\left\vert \func{Im}%
\mu \right\vert \pi }$ $\left( M>0\right) $ by using a similar method in
[11]. Thus, on $\mu \in C_{n}$ or $\mu \in \Gamma _{N_{0}},$ from (15) and
(16), we have%
\begin{equation*}
\frac{\Xi \left( \mu \right) }{\Xi _{0}\left( \mu \right) }=1+\frac{1}{\mu }%
\left[ \left( \frac{\beta _{1}^{+}}{\beta _{2}^{+}}+\frac{\alpha _{1}^{+}}{%
\alpha _{2}^{+}}+U^{+}\left( \mu \right) +V^{+}\left( \mu \right) \right)
\cot \mu \pi \right.
\end{equation*}%
\begin{equation*}
\left. +\frac{\alpha _{2}^{-}}{\alpha _{2}^{+}}+\frac{\beta _{2}^{-}}{\beta
_{2}^{+}}+U^{-}\left( \mu \right) +V^{-}\left( \mu \right) \right] +O\left( 
\frac{1}{\mu ^{2}}\right) .
\end{equation*}%
Expanding $\ln \frac{\Xi \left( \mu \right) }{\Xi _{0}\left( \mu \right) }$
by the Maclaurin formula, we find that%
\begin{equation*}
\ln \frac{\Xi \left( \mu \right) }{\Xi _{0}\left( \mu \right) }=\frac{1}{\mu 
}\left[ \left( \frac{\beta _{1}^{+}}{\beta _{2}^{+}}+\frac{\alpha _{1}^{+}}{%
\alpha _{2}^{+}}+U^{+}\left( \mu \right) +V^{+}\left( \mu \right) \right)
\cot \mu \pi \right.
\end{equation*}%
\begin{equation*}
\left. +\frac{\alpha _{2}^{-}}{\alpha _{2}^{+}}+\frac{\beta _{2}^{-}}{\beta
_{2}^{+}}+U^{-}\left( \mu \right) +V^{-}\left( \mu \right) \right]
\end{equation*}%
\begin{equation*}
-\frac{1}{2\mu ^{2}}\left[ \left( \frac{\beta _{1}^{+}}{\beta _{2}^{+}}+%
\frac{\alpha _{1}^{+}}{\alpha _{2}^{+}}+U^{+}\left( \mu \right) +V^{+}\left(
\mu \right) \right) ^{2}\cot ^{2}\mu \pi \right.
\end{equation*}%
\begin{equation*}
+\left( \frac{\alpha _{2}^{-}}{\alpha _{2}^{+}}+\frac{\beta _{2}^{-}}{\beta
_{2}^{+}}+U^{-}\left( \mu \right) +V^{-}\left( \mu \right) \right) ^{2}
\end{equation*}%
\begin{equation*}
+2\left( \frac{\beta _{1}^{+}}{\beta _{2}^{+}}+\frac{\alpha _{1}^{+}}{\alpha
_{2}^{+}}+U^{+}\left( \mu \right) +V^{+}\left( \mu \right) \right)
\end{equation*}%
\begin{equation*}
\left. \times \left( \frac{\alpha _{2}^{-}}{\alpha _{2}^{+}}+\frac{\beta
_{2}^{-}}{\beta _{2}^{+}}+U^{-}\left( \mu \right) +V^{-}\left( \mu \right)
\right) \cot \mu \pi \right] +O\left( \frac{1}{\mu ^{3}}\right) .
\end{equation*}%
Using the well-known Rouche Theorem, we get that $\Xi \left( \mu \right) $
has the same number of zeros inside $\Gamma _{N_{0}}$ as $\Xi _{0}\left( \mu
\right) $ (see [12]). It is easy to prove that the spectrum of problem
(1)-(5) is 
\begin{equation*}
\mu _{n}\sim \mu _{n}^{0}+O\left( \frac{1}{n}\right) \text{ as\ }\left\vert
n\right\vert \rightarrow \infty .
\end{equation*}%
Next, we present the more exact asymptotic distribution of the spectrum.
Using the residue theorem, we have%
\begin{equation*}
\mu _{n}-\mu _{n}^{0}=-\frac{1}{2\pi i}\doint\limits_{C_{n}}\ln \frac{\Xi
\left( \mu \right) }{\Xi _{0}\left( \mu \right) }d\mu
\end{equation*}%
\begin{equation*}
=-\frac{1}{2\pi i}\doint\limits_{C_{n}}\left( \frac{\beta _{1}^{+}}{\beta
_{2}^{+}}+\frac{\alpha _{1}^{+}}{\alpha _{2}^{+}}+U^{+}\left( \mu \right)
+V^{+}\left( \mu \right) \right) \frac{\cot \left( \mu \pi \right) }{\mu }%
d\mu
\end{equation*}%
\begin{equation*}
-\frac{1}{2\pi i}\doint\limits_{C_{n}}\left( \frac{\alpha _{2}^{-}}{\alpha
_{2}^{+}}+\frac{\beta _{2}^{-}}{\beta _{2}^{+}}+U^{-}\left( \mu \right)
+V^{-}\left( \mu \right) \right) \frac{d\mu }{\mu }+O\left( \frac{1}{n^{2}}%
\right)
\end{equation*}%
\begin{equation*}
=-\frac{1}{\mu _{n}^{0}\pi }\left( \frac{\beta _{1}^{+}}{\beta _{2}^{+}}+%
\frac{\alpha _{1}^{+}}{\alpha _{2}^{+}}+U^{+}\left( n\right) +V^{+}\left(
n\right) \right) +O\left( \frac{1}{n^{2}}\right) .
\end{equation*}%
Thus we have proven the following theorem.

\begin{theorem}
The spectrum of the problem (1)-(5) has the%
\begin{equation}
\mu _{n}=\mu _{n}^{0}-\frac{1}{\mu _{n}^{0}\pi }\left( \frac{\beta _{1}^{+}}{%
\beta _{2}^{+}}+\frac{\alpha _{1}^{+}}{\alpha _{2}^{+}}+U^{+}\left( \mu
_{n}^{0}\right) +V^{+}\left( \mu _{n}^{0}\right) \right) +O\left( \frac{1}{%
\left( \mu _{n}^{0}\right) ^{2}}\right)  \tag{17}
\end{equation}%
asymptotic distribution for sufficiently large $\left\vert n\right\vert .$
\end{theorem}

Finally, following [12,13], we will get regularized trace formula for the
problem (1)-(5).

The asymptotic formula (17) for the eigenvalues implies that for all
sufficiently large $N_{0},$ the numbers $\mu _{n}$ with $\left\vert
n\right\vert \leq N_{0}$ are inside $\Gamma _{N_{0}},$ and the numbers $\mu
_{n}$ with $\left\vert n\right\vert >N_{0}$ are outside $\Gamma _{N_{0}}.$
It follows that%
\begin{equation*}
\dsum\limits_{\Gamma _{n}}\left( \mu _{n}^{2}-\left( \mu _{n}^{0}\right)
^{2}\right) =\mu _{-0}^{2}+\mu _{0}^{2}+\sum_{0\neq n=-N_{0}}^{N_{0}}\left(
\mu _{n}^{2}-\left( \mu _{n}^{0}\right) ^{2}\right) =-\frac{1}{\pi i}%
\doint\limits_{\Gamma _{n}}\mu \ln \frac{\Xi \left( \mu \right) }{\Xi
_{0}\left( \mu \right) }d\mu
\end{equation*}%
\begin{equation*}
=-\frac{1}{\pi i}\doint\limits_{\Gamma _{n}}\left( \frac{\beta _{1}^{+}}{%
\beta _{2}^{+}}+\frac{\alpha _{1}^{+}}{\alpha _{2}^{+}}+U^{+}\left( \mu
\right) +V^{+}\left( \mu \right) \right) \cot \left( \mu \pi \right) d\mu
\end{equation*}%
\begin{equation*}
-\frac{1}{\pi i}\doint\limits_{\Gamma _{n}}\left( \frac{\alpha _{2}^{-}}{%
\alpha _{2}^{+}}+\frac{\beta _{2}^{-}}{\beta _{2}^{+}}+U^{-}\left( \mu
\right) +V^{-}\left( \mu \right) \right) d\mu
\end{equation*}%
\begin{equation*}
+\frac{1}{2\pi i}\doint\limits_{\Gamma _{n}}\left( \frac{\beta _{1}^{+}}{%
\beta _{2}^{+}}+\frac{\alpha _{1}^{+}}{\alpha _{2}^{+}}+U^{+}\left( \mu
\right) +V^{+}\left( \mu \right) \right) ^{2}\frac{\cot ^{2}\left( \mu \pi
\right) }{\mu }d\mu
\end{equation*}%
\begin{equation*}
+\frac{1}{2\pi i}\doint\limits_{\Gamma _{n}}\left( \frac{\alpha _{2}^{-}}{%
\alpha _{2}^{+}}+\frac{\beta _{2}^{-}}{\beta _{2}^{+}}+U^{-}\left( \mu
\right) +V^{-}\left( \mu \right) \right) ^{2}\frac{d\mu }{\mu }
\end{equation*}%
\begin{equation*}
+\frac{1}{\pi i}\doint\limits_{\Gamma _{n}}\left( \frac{\beta _{1}^{+}}{%
\beta _{2}^{+}}+\frac{\alpha _{1}^{+}}{\alpha _{2}^{+}}+U^{+}\left( \mu
\right) +V^{+}\left( \mu \right) \right)
\end{equation*}%
\begin{equation*}
\times \left( \frac{\alpha _{2}^{-}}{\alpha _{2}^{+}}+\frac{\beta _{2}^{-}}{%
\beta _{2}^{+}}+U^{-}\left( \mu \right) +V^{-}\left( \mu \right) \right) 
\frac{\cot \left( \mu \pi \right) }{\mu }d\mu +O\left( \frac{1}{N_{0}}%
\right) ,
\end{equation*}%
by calculations, which implies that%
\begin{equation*}
\mu _{-0}^{2}+\mu _{0}^{2}+\sum_{0\neq n=-N_{0}}^{N_{0}}\left( \mu
_{n}^{2}-\left( \mu _{n}^{0}\right) ^{2}+\frac{4}{\pi }\left( \frac{\beta
_{1}^{+}}{\beta _{2}^{+}}+\frac{\alpha _{1}^{+}}{\alpha _{2}^{+}}%
+U^{+}\left( n\right) +V^{+}\left( n\right) \right) \right)
\end{equation*}%
\begin{equation*}
=-\frac{2}{\pi }\left( \frac{\beta _{1}^{+}}{\beta _{2}^{+}}-\frac{\alpha
_{1}^{+}}{\alpha _{2}^{+}}+U^{+}\left( 0\right) +V^{+}\left( 0\right)
\right) -\left( \frac{\beta _{1}^{+}}{\beta _{2}^{+}}+\frac{\alpha _{1}^{+}}{%
\alpha _{2}^{+}}+U^{+}\left( 0\right) +V^{+}\left( 0\right) \right) ^{2}
\end{equation*}%
\begin{equation}
+\left( \frac{\alpha _{2}^{-}}{\alpha _{2}^{+}}+\frac{\beta _{2}^{-}}{\beta
_{2}^{+}}+U^{-}\left( 0\right) +V^{-}\left( 0\right) \right) ^{2}+O\left( 
\frac{1}{N_{0}}\right) .  \tag{18}
\end{equation}%
Passing to the limit as $N_{0}\rightarrow \infty $ in (18), we have%
\begin{equation*}
\mu _{-0}^{2}+\mu _{0}^{2}+\sum_{0\neq n=-\infty }^{+\infty }\left( \mu
_{n}^{2}-\left( \mu _{n}^{0}\right) ^{2}+\frac{4}{\pi }\left( \frac{\beta
_{1}^{+}}{\beta _{2}^{+}}+\frac{\alpha _{1}^{+}}{\alpha _{2}^{+}}%
+U^{+}\left( \mu _{n}^{0}\right) +V^{+}\left( \mu _{n}^{0}\right) \right)
\right)
\end{equation*}%
\begin{equation}
=-\frac{2}{\pi }\left( \frac{\beta _{1}^{+}}{\beta _{2}^{+}}+\frac{\alpha
_{1}^{+}}{\alpha _{2}^{+}}+U^{+}\left( 0\right) +V^{+}\left( 0\right)
\right) -\left( \frac{\beta _{1}^{+}}{\beta _{2}^{+}}+\frac{\alpha _{1}^{+}}{%
\alpha _{2}^{+}}+U^{+}\left( 0\right) +V^{+}\left( 0\right) \right)
^{2}+\left( \frac{\alpha _{2}^{-}}{\alpha _{2}^{+}}+\frac{\beta _{2}^{-}}{%
\beta _{2}^{+}}\right) ^{2}.  \tag{19}
\end{equation}%
The series on the left side of (19) is called the regularized trace of the
problem (1)-(5).

We want to note that the trace formulas for different types of boundary
value problems with retarded argument obtained in [12-17] and approximate
calculation of the eigenvalues of the problem (1)-(5) can also be obtained
via formula (19) (see [18-20]).

\section{The inverse problem}

Inverse nodal problems for differential operators with or without retarded
argument were investigated by a number of authors( see [13, 16, 17, 21-27]
and the references therein). In this chapter, following [12], we deal with
inverse spectral analysis of the problem (1)-(5) using the nodal points
(zeros) of its eigenfunctions.

Let us rewrite the equation (8) as%
\begin{equation*}
\varphi _{1}(t,\mu )=\left( \mu \alpha _{2}^{+}+\alpha _{2}^{-}\right) \cos
\mu t-\frac{1}{\mu }\left( \mu \alpha _{1}^{+}+\alpha _{1}^{-}\right) \sin
\mu t
\end{equation*}%
\begin{equation*}
-\frac{\alpha _{2}^{+}}{2}\int\limits_{0}^{t}q(\tau )\left[ \sin \mu \left(
t-2\tau +\Delta (\tau )\right) +\sin \mu \left( t-\Delta (\tau )\right) %
\right] d\tau
\end{equation*}%
\begin{equation*}
-\frac{\alpha _{2}^{-}}{2\mu }\int\limits_{0}^{t}q(\tau )\left[ \sin \mu
\left( t-2\tau +\Delta (\tau )\right) +\sin \mu \left( t-\Delta (\tau
)\right) \right] d\tau
\end{equation*}%
\begin{equation*}
-\frac{\alpha _{1}^{+}}{2\mu }\int\limits_{0}^{t}q(\tau )\left[ \cos \mu
\left( t-2\tau +\Delta (\tau )\right) -\cos \mu \left( t-\Delta (\tau
)\right) \right] d\tau +O\left( \frac{1}{\mu ^{2}}\right) ,
\end{equation*}%
and using the fact that 
\begin{equation*}
\int\limits_{0}^{t}q(\tau )\sin \mu \left( 2\tau -\Delta (\tau )\right)
d\tau =\int\limits_{0}^{t}q(\tau )\cos \mu \left( 2\tau -\Delta (\tau
)\right) d\tau =O\left( \frac{1}{\mu }\right) ,
\end{equation*}%
(see [3, Lemma 2.3.3]) it yields that%
\begin{equation*}
\varphi _{1}(t,\mu )=\mu \alpha _{2}^{+}\cos \mu t+\alpha _{2}^{-}\cos \mu
t-\alpha _{1}^{+}\sin \mu t-\frac{\alpha _{1}^{-}}{\mu }\sin \mu t
\end{equation*}%
\begin{equation*}
-\frac{\alpha _{2}^{+}\cos \mu t}{2}\int\limits_{0}^{t}q(\tau )\sin \mu
\left( \Delta (\tau )\right) d\tau -\frac{\alpha _{2}^{+}\sin \mu t}{2}%
\int\limits_{0}^{t}q(\tau )\cos \mu \left( \Delta (\tau )\right) d\tau
\end{equation*}%
\begin{equation*}
-\frac{\alpha _{2}^{-}\cos \mu t}{2\mu }\int\limits_{0}^{t}q(\tau )\sin \mu
\left( \Delta (\tau )\right) d\tau -\frac{\alpha _{2}^{-}\sin \mu t}{2\mu }%
\int\limits_{0}^{t}q(\tau )\cos \mu \left( \Delta (\tau )\right) d\tau
\end{equation*}%
\begin{equation*}
+\frac{\alpha _{1}^{+}\cos \mu t}{2\mu }\int\limits_{0}^{t}q(\tau )\sin \mu
\left( \Delta (\tau )\right) d\tau +\frac{\alpha _{1}^{+}\sin \mu t}{2\mu }%
\int\limits_{0}^{t}q(\tau )\cos \mu \left( \Delta (\tau )\right) d\tau
+O\left( \frac{1}{\mu ^{2}}\right) ,
\end{equation*}%
\begin{equation*}
=\mu _{n}\alpha _{2}^{+}\cos \mu _{n}t+\alpha _{2}^{-}\cos \mu _{n}t-\alpha
_{1}^{+}\sin \mu _{n}t
\end{equation*}%
\begin{equation*}
-\frac{\alpha _{1}^{-}}{\mu _{n}}\sin \mu _{n}t-\frac{\alpha _{2}^{+}}{2}%
\int\limits_{0}^{t}q(\tau )\sin \mu _{n}\left( t-\Delta (\tau )\right) d\tau
-\frac{\alpha _{2}^{-}}{2\mu _{n}}\int\limits_{0}^{t}q(\tau )\sin \mu
_{n}\left( t-\Delta (\tau )\right) d\tau
\end{equation*}%
\begin{equation*}
+\frac{\alpha _{1}^{+}}{2\mu _{n}}\int\limits_{0}^{t}q(\tau )\cos \mu
_{n}\left( t-\Delta (\tau )\right) d\tau +O\left( \frac{1}{\mu ^{2}}\right)
\end{equation*}%
\begin{equation*}
=\mu _{n}\alpha _{2}^{+}\cos \mu _{n}t+\alpha _{2}^{-}\cos \mu _{n}t-\alpha
_{1}^{+}\sin \mu _{n}t-\frac{\alpha _{1}^{-}}{\mu _{n}}\sin \mu _{n}t
\end{equation*}%
\begin{equation*}
-\frac{\alpha _{2}^{+}\sin \mu _{n}t}{2}\int\limits_{0}^{t}q(\tau )\cos
\left( \mu _{n}\Delta (\tau )\right) d\tau -\frac{\alpha _{2}^{+}\cos \mu
_{n}t}{2}\int\limits_{0}^{t}q(\tau )\sin \left( \mu _{n}\Delta (\tau
)\right) d\tau
\end{equation*}%
\begin{equation*}
-\frac{\alpha _{2}^{-}\sin \mu _{n}t}{2\mu _{n}}\int\limits_{0}^{t}q(\tau
)\cos \left( \mu _{n}\Delta (\tau )\right) d\tau -\frac{\alpha _{2}^{-}\cos
\mu _{n}t}{2\mu _{n}}\int\limits_{0}^{t}q(\tau )\sin \left( \mu _{n}\Delta
(\tau )\right) d\tau
\end{equation*}%
\begin{equation*}
+\frac{\alpha _{1}^{+}\sin \mu _{n}t}{2\mu _{n}}\int\limits_{0}^{t}q(\tau
)\sin \left( \mu _{n}\Delta (\tau )\right) d\tau +\frac{\alpha _{1}^{+}\cos
\mu _{n}t}{2\mu _{n}}\int\limits_{0}^{t}q(\tau )\cos \left( \mu _{n}\Delta
(\tau )\right) d\tau +O\left( \frac{1}{\mu _{n}^{2}}\right) .
\end{equation*}%
Let us assume that $t_{n}^{j}$ are the nodal points of the eigenfunction $%
\varphi \left( t,\mu _{n}\right) .$ Taking $\sin \left( \mu _{n}t\right)
\neq 0$ into account for sufficiently large $n,$ we get 
\begin{equation*}
T\left( \mu _{n},t\right) \cot \mu _{n}t=\alpha _{1}^{+}+\frac{\alpha
_{1}^{-}}{\mu _{n}}+\frac{\alpha _{2}^{+}}{2}\int\limits_{0}^{t}q(\tau )\cos
\left( \mu _{n}\Delta (\tau )\right) d\tau
\end{equation*}%
\begin{equation*}
+\frac{\alpha _{2}^{-}}{2\mu _{n}}\int\limits_{0}^{t}q(\tau )\cos \left( \mu
_{n}\Delta (\tau )\right) d\tau -\frac{\alpha _{1}^{+}}{2\mu _{n}}%
\int\limits_{0}^{t}q(\tau )\sin \left( \mu _{n}\Delta (\tau )\right) d\tau
+O\left( \frac{1}{\mu _{n}^{2}}\right) .
\end{equation*}%
Here 
\begin{equation*}
T\left( \mu _{n},t\right) =\left\{ \mu _{n}\alpha _{2}^{+}+\alpha _{2}^{-}-%
\frac{\alpha _{2}^{+}}{2}\int\limits_{0}^{t}q(\tau )\sin \left( \mu
_{n}\Delta (\tau )\right) d\tau \right.
\end{equation*}%
\begin{equation*}
\left. -\frac{\alpha _{2}^{-}}{2\mu _{n}}\int\limits_{0}^{t}q(\tau )\sin
\left( \mu _{n}\Delta (\tau )\right) d\tau +\frac{\alpha _{1}^{+}}{2\mu _{n}}%
\int\limits_{0}^{t}q(\tau )\cos \left( \mu _{n}\Delta (\tau )\right) d\tau
\right\}
\end{equation*}%
and it follows easily that 
\begin{equation}
\tan \left( \mu _{n}t+\frac{\pi }{2}\right) =\frac{\alpha _{1}^{+}}{T\left(
\mu _{n},t\right) }+\frac{\alpha _{2}^{+}}{2\mu _{n}T\left( \mu
_{n},t\right) }\int\limits_{0}^{t}q(\tau )\cos \left( \mu _{n}\Delta (\tau
)\right) d\tau +O\left( \frac{1}{\mu _{n}^{3}}\right) .  \tag{20}
\end{equation}%
Thus, solving the equation (20), one obtains%
\begin{equation}
t_{n}^{j}=\frac{\left( j-\frac{1}{2}\right) \pi }{\mu _{n}}+\frac{\alpha
_{1}^{+}}{\mu _{n}T\left( \mu _{n},t_{n}^{j}\right) }+\frac{\alpha _{2}^{+}}{%
2\mu _{n}T\left( \mu _{n},t_{n}^{j}\right) }\int\limits_{0}^{t_{n}^{j}}q(%
\tau )\cos \left( \mu _{n}\Delta (\tau )\right) d\tau +O\left( \frac{1}{\mu
_{n}^{3}}\right)  \tag{21}
\end{equation}%
Note that%
\begin{equation}
\mu _{n}^{-1}=\frac{1}{\mu _{n}^{0}}-\frac{\left( \frac{\beta _{1}^{+}}{%
\beta _{2}^{+}}+\frac{\alpha _{1}^{+}}{\alpha _{2}^{+}}+U^{+}\left( n\right)
\right) }{\left( \mu _{n}^{0}\right) ^{3}\pi }+O\left( \frac{1}{n^{4}}%
\right) .  \tag{22}
\end{equation}%
Substituting (22) into (21) we have 
\begin{equation*}
t_{n}^{j}=\left( j-\frac{1}{2}\right) \pi \left[ \frac{1}{\mu _{n}^{0}}-%
\frac{\left( \frac{\beta _{1}^{+}}{\beta _{2}^{+}}+\frac{\alpha _{1}^{+}}{%
\alpha _{2}^{+}}+U^{+}\left( \mu _{n}^{0}\right) \right) }{\left( \mu
_{n}^{0}\right) ^{3}\pi }\right]
\end{equation*}%
\begin{equation}
+\frac{\alpha _{1}^{+}}{\mu _{n}^{0}T_{0}\left( n\right) }+\frac{\alpha
_{2}^{+}}{2\mu _{n}^{0}T_{0}\left( n\right) }\int\limits_{0}^{\frac{j\pi }{n}%
}q(\tau )\cos \left( \mu _{n}^{0}\Delta (\tau )\right) d\tau +O\left( \frac{1%
}{\left( \mu _{n}^{0}\right) ^{3}}\right) ,\text{ \ \ \ }j=\overline{1,\left[
\frac{n}{2}\right] }.  \tag{23}
\end{equation}%
Here $T_{0}\left( n\right) =T\left( \mu _{n}^{0},\frac{j\pi }{n}\right) .$
Similarly, from (9), we get 
\begin{align*}
\left( \prod\limits_{i=1}^{m}\delta _{i}\right) \varphi _{2}(t,\mu _{n})&
=\mu _{n}\alpha _{2}^{+}\cos \left( \mu _{n}t\right) +\alpha _{2}^{-}\cos
\left( \mu _{n}t\right) -\alpha _{1}^{+}\sin \left( \mu _{n}t\right) \\
& -\frac{\alpha _{2}^{+}}{2}\int\limits_{0}^{{t}}q\left( \tau \right) \sin
\mu _{n}\left( t-\Delta \left( \tau \right) \right) d\tau +O\left( \frac{1}{%
\mu _{n}}\right) ,
\end{align*}%
and%
\begin{equation*}
\left( \prod\limits_{i=1}^{m}\delta _{i}\right) \varphi _{2}(t,\mu _{n})=\mu
_{n}\alpha _{2}^{+}\cos \left( \mu _{n}t\right) +\alpha _{2}^{-}\cos \left(
\mu _{n}t\right) -\alpha _{1}^{+}\sin \left( \mu _{n}t\right)
\end{equation*}%
\begin{equation*}
+\frac{\alpha _{2}^{+}\cos \left( \mu _{n}t\right) }{2}\int\limits_{0}^{{t}%
}q\left( \tau \right) \sin \left( \mu _{n}\Delta \left( \tau \right) \right)
d\tau -\frac{\alpha _{2}^{+}\sin \left( \mu _{n}t\right) }{2}%
\int\limits_{0}^{{t}}q\left( \tau \right) \cos \left( \mu _{n}\Delta \left(
\tau \right) \right) d\tau +O\left( \frac{1}{\mu _{n}}\right) .
\end{equation*}%
For nodal points of $\varphi _{2}(t,\mu _{n})$, we have the equality%
\begin{equation*}
0=\alpha _{2}^{+}\cos \left( \mu _{n}t\right) +\frac{\alpha _{2}^{-}}{\mu
_{n}}\cos \left( \mu _{n}t\right) -\frac{\alpha _{1}^{+}}{\mu _{n}}\sin
\left( \mu _{n}t\right)
\end{equation*}%
\begin{equation*}
+\frac{\alpha _{2}^{+}\cos \left( \mu _{n}t\right) }{2\mu _{n}}%
\int\limits_{0}^{{t}}q\left( \tau \right) \sin \left( \mu _{n}\Delta \left(
\tau \right) \right) d\tau -\frac{\alpha _{2}^{+}\sin \left( \mu
_{n}t\right) }{2\mu _{n}}\int\limits_{0}^{{t}}q\left( \tau \right) \cos
\left( \mu _{n}\Delta \left( \tau \right) \right) d\tau +O\left( \frac{1}{%
\mu _{n}^{2}}\right) .
\end{equation*}%
Again, taking $\sin \left( \mu _{n}t\right) \neq 0$ into account for
sufficiently large $n,$ we get%
\begin{equation*}
0=\alpha _{2}^{+}\cot \left( \mu _{n}t\right) +\frac{\alpha _{2}^{-}}{\mu
_{n}}\cot \left( \mu _{n}t\right) -\frac{\alpha _{1}^{+}}{\mu _{n}}
\end{equation*}%
\begin{equation*}
+\frac{\alpha _{2}^{+}\cot \left( \mu _{n}t\right) }{2\mu _{n}}%
\int\limits_{0}^{{t}}q\left( \tau \right) \sin \left( \mu _{n}\Delta \left(
\tau \right) \right) d\tau -\frac{\alpha _{2}^{+}}{2\mu _{n}}%
\int\limits_{0}^{{t}}q\left( \tau \right) \cos \left( \mu _{n}\Delta \left(
\tau \right) \right) d\tau +O\left( \frac{1}{\mu _{n}^{2}}\right) ,
\end{equation*}%
and%
\begin{equation}
\tan \left( \mu _{n}t+\frac{\pi }{2}\right) =\frac{\alpha _{1}^{+}}{\alpha
_{2}^{+}\mu _{n}}+\frac{1}{2\mu _{n}}\int\limits_{0}^{{t}}q\left( \tau
\right) \cos \left( \mu _{n}\Delta \left( \tau \right) \right) d\tau
+O\left( \frac{1}{\mu _{n}^{2}}\right) .  \tag{24}
\end{equation}%
Thus, solving the equation (24), one obtains%
\begin{equation}
t_{n}^{j}=\frac{\left( j-\frac{1}{2}\right) \pi }{\mu _{n}}+\frac{\alpha
_{1}^{+}}{\alpha _{2}^{+}\mu _{n}^{2}}+\frac{1}{2\mu _{n}^{2}}%
\int\limits_{0}^{t_{n}^{j}}q\left( \tau \right) \cos \left( \mu _{n}\Delta
\left( \tau \right) \right) d\tau +O\left( \frac{1}{\mu _{n}^{3}}\right) . 
\tag{25}
\end{equation}%
Note that%
\begin{equation}
\mu _{n}^{-2}=\frac{1}{\left( \mu _{n}^{0}\right) ^{2}}+O\left( \frac{1}{%
n^{4}}\right) .  \tag{26}
\end{equation}%
Substituting (26) into (25) we have 
\begin{equation*}
t_{n}^{j}=\left( j-\frac{1}{2}\right) \pi \left( \frac{1}{\mu _{n}^{0}}-%
\frac{\left( \frac{\beta _{1}^{+}}{\beta _{2}^{+}}+\frac{\alpha _{1}^{+}}{%
\alpha _{2}^{+}}+U^{+}\left( \mu _{n}^{0}\right) \right) }{\left( \mu
_{n}^{0}\right) ^{3}\pi }\right)
\end{equation*}%
\begin{equation}
+\frac{1}{\left( \mu _{n}^{0}\right) ^{2}}\left( \frac{\alpha _{1}^{+}}{%
\alpha _{2}^{+}}+\frac{1}{2}\int\limits_{0}^{\frac{j\pi }{n}}q\left( \tau
\right) \cos \left( \mu _{n}^{0}\Delta \left( \tau \right) \right) d\tau
\right) +O\left( \frac{1}{\left( \mu _{n}^{0}\right) ^{3}}\right) ,\text{ \
\ \ }j=\overline{\left[ \frac{n}{2}\right] +1,n}.  \tag{27}
\end{equation}%
Thus we have proven the following theorem:

\begin{theorem}
For sufficiently large $n$, we have the \ formulas (23) and (27) of the
nodal points for the problem (1)-(5).
\end{theorem}

We see that there exists $N_{0}$ such that for all $n>N_{0}$ the
eigenfunction $\varphi \left( t,\mu _{n}\right) $ of the problem has exactly 
$n$ simple nodes in the interval $\left( 0,\pi \right) .$ The set $\Lambda
=\left\{ t_{n}^{j}\right\} $ is called the nodal set of the problem (1)-(5).
We also define the function $j_{n}\left( t\right) $ to be the largest index $%
j$ such that $0\leq t_{n}^{j}\leq t.$ Thus, $j=j_{n}\left( t\right) $ if and
only if $t\in \left[ t_{n}^{j},t_{n}^{j+1}\right) .$

\begin{theorem}
For each $t\in \left[ 0,\pi \right] $, let $\left\{ t_{n}^{j}\right\}
\subset \Lambda $ be chosen such that $\lim_{n\rightarrow \infty
}t_{n}^{j}=t.$ Then the following finite limit exists and corresponding
equality holds:%
\begin{equation}
\lim_{n\rightarrow \infty }\left( \mu _{n}^{0}\right) ^{2}\left[ t_{n}^{j}-%
\frac{\left( j-\frac{1}{2}\right) \pi }{\mu _{n}^{0}}\right] =f\left(
t\right)  \tag{28}
\end{equation}%
and%
\begin{equation}
f(t)=\left\{ 
\begin{array}{cc}
\frac{\left( \frac{\beta _{1}^{+}}{\beta _{2}^{+}}+\frac{\alpha _{1}^{+}}{%
\alpha _{2}^{+}}+U^{+}\left( 0\right) \right) t}{\pi }-\frac{\alpha _{1}^{+}%
}{\alpha _{2}^{+}}-\frac{1}{2}\int\limits_{0}^{t}q\left( \tau \right) d\tau ,
& \Delta \left( \tau \right) =0, \\ 
\frac{\left( \frac{\beta _{1}^{+}}{\beta _{2}^{+}}+\frac{\alpha _{1}^{+}}{%
\alpha _{2}^{+}}\right) t}{\pi }-\frac{\alpha _{1}^{+}}{\alpha _{2}^{+}}, & 
\Delta \left( \tau \right) \neq 0.%
\end{array}%
\right.  \tag{29}
\end{equation}
\end{theorem}

\begin{proof}
Using the formulas (23) and (27) for nodal points and the fact that $%
\lim_{n\rightarrow \infty }t_{n}^{j}=t,$ it follows that as $n\rightarrow
\infty $ the limits of left-hand side in (28) exists and Eq. (29) holds.
Thus, the proof is completed.
\end{proof}

Now, we can construct the potential function $q\left( t\right) $ via
following theorem.

\begin{theorem}
Let $\Lambda ^{0}=\left\{ t_{n_{k}}^{j}\right\} $ and $\Lambda ^{0}\subset
\Lambda $ be a subset of nodal points which satisfy $\left\{
t_{n_{k}}^{j}\right\} $ is dense in $\left( 0,\pi \right) .$ For each $t\in %
\left[ 0,\pi \right] $ choose a sequence $\left\{ t_{n}^{j}\right\} \subset
\Lambda ^{0}$ such that $\lim_{n\rightarrow \infty }t_{n}^{j}=t.$ If $\Delta
\left( \tau \right) =0,$ then the function $q\left( t\right) $ can be
written as%
\begin{equation*}
q(t)=\frac{2}{\pi }\left[ U^{+}\left( 0\right) +f(\pi )-f(0)\right]
-2f^{\prime }(t).
\end{equation*}
\end{theorem}

Here, $f(t)$ is defined by (29).

Now, we will give a few numerical examples to verify and illustrate the
results. For each example we shall display spectrum, the regularized trace
and the solution of inverse nodal problem. In each example we will take $%
n=40 $.

\textbf{Example 1.} Consider the following problem $L_{1}:$

\begin{equation}
y^{\prime \prime }(t)+ty(\frac{t}{2})+\mu ^{2}y(t)=0,\text{ }t\in \Omega =%
\left[ 0,1\right) \cup \left( 1,\pi \right] ;  \tag{30}
\end{equation}%
\begin{equation}
y(0)+8y^{\prime }(0)+\mu y^{\prime }(0)=0,  \tag{31}
\end{equation}%
\begin{equation}
y\left( \pi \right) +\frac{1}{10}y^{\prime }\left( \pi \right) +\mu
y^{\prime }\left( \pi \right) =0,  \tag{32}
\end{equation}%
\begin{equation}
y\left( 1-\right) =y(1+),  \tag{33}
\end{equation}%
\begin{equation}
y^{\prime }\left( 1-\right) =y^{\prime }(1+).  \tag{34}
\end{equation}%
Namely; $\Delta (t)=\frac{t}{2},$ $\alpha _{1}^{+}=0,$ $\alpha _{1}^{-}=1,$ $%
\alpha _{2}^{-}=-8,$ $\alpha _{2}^{+}=-1,$ $\beta _{1}^{+}=0,$ $\beta
_{2}^{+}=-1,$ $\beta _{1}^{-}=1,$ $\beta _{2}^{-}=-0.1$ and $\delta _{1}=1$.
For the spectrum of the problem $L_{1},$ since $U^{+}\left( 39\right)
=0.00065746219$ and $V^{+}\left( 39\right) =0.02670511654$.from Theorem 2.1,
we have

\begin{equation*}
\mu _{40}=38.9997766723+O\left( 0.00065746219\right) .
\end{equation*}%
From (19), since $U^{+}\left( 0\right) =V^{+}\left( 0\right) =4.93480220054$
for the regularized trace ($tr$=trace), we obtain $trL_{1}=58.3910062461$.
Now, consider the differential equation (take $\Delta =0$ in (30))%
\begin{equation*}
y^{\prime \prime }(t)+q(t)y(t)+\mu ^{2}y(t)=0,\text{\ }t\in \Omega
\end{equation*}%
with the boundary conditions (31)-(32) and transmission conditions
(33)-(34). Thus, from Theorem 3.1 and Theorem 3.2, for the solution of
inverse problem we find $q\left( t\right) =t$ in $\Omega $.

\textbf{Example 2.} Consider the problem $L_{2}:$%
\begin{equation*}
y^{\prime \prime }(t)+q(t)y(t)+\mu ^{2}y(t)=0,\text{ }t\in \Omega =\left[
0,1.5\right) \cup \left( 1.5,2\right) \cup \left( 2,\pi \right]
\end{equation*}%
\begin{eqnarray*}
\left( 3\mu +2\right) y(0)-\left( 7\mu +4\right) y^{\prime }\left( 0\right)
&=&0, \\
\left( \mu -5\right) y\left( \pi \right) -\left( \mu +0.3\right) y^{\prime
}\left( \pi \right) &=&0,
\end{eqnarray*}%
\begin{equation*}
y\left( 1.5-\right) -2y(1.5+)=0,
\end{equation*}%
\begin{equation*}
y^{\prime }\left( 1.5-\right) -2y^{\prime }(1.5+)=0,
\end{equation*}%
\begin{equation*}
y\left( 2-\right) -8y(2+)=0,
\end{equation*}%
\begin{equation*}
y^{\prime }\left( 2-\right) -8y^{\prime }(2+)=0.
\end{equation*}%
Here , $\Delta (t)=0$, $\alpha _{1}^{+}=3,$ $\alpha _{1}^{-}=2,$ $\alpha
_{2}^{-}=4,$ $\alpha _{2}^{+}=7,$ $\beta _{1}^{+}=1,$ $\beta _{1}^{-}=-5,$ $%
\beta _{2}^{+}=1,$ $\beta _{2}^{-}=0.3$, $\delta _{1}=2$ and $\delta _{2}=8$%
. If $U^{+}\left( 39\right) =11.0703463164$ and $V^{+}\left( 39\right)
=0.00181958345$ then we have $trL_{2}=-569.751286593$ and

\begin{equation*}
\mu _{40}=37.9930198832+O\left( 0.00065746219\right) .
\end{equation*}%
Consequently, from Theorem 3.1 and Theorem 3.2, for the solution of inverse
problem we have $q\left( t\right) =e^{t}$ in $\Omega $.


\end{document}